\documentclass{amsart}
\usepackage{amsmath}
\usepackage{amsthm}
\usepackage{amsfonts}
\usepackage{mathrsfs}
\usepackage[margin=1.4in]{geometry}
\usepackage{verbatim}
\usepackage{url}
\usepackage{color}

\newtheorem*{theoremnonum}{Theorem}
\newtheorem{theorem}{Theorem}[section]
  \newtheorem{definition}[theorem]{Definition}
\newtheorem{lemma}[theorem]{Lemma}

\newtheorem{cor}[theorem]{Corollary}

\theoremstyle{definition}

  \newtheorem*{remarks}{Remarks}

\newcommand\C{\mathbb{C}}
\newcommand\F{\mathbb{F}}

\newcommand\Z{\mathbb{Z}}

\newcommand\somedesign{D}

\newcommand\setsep{:}

\newcommand\defeq{:=}

  \newcommand{\ra}{\rightarrow}
  \newcommand{\divs}{\,|\,}

\newcommand\PSL{\mathrm{PSL}}

\newcommand\dzoharmd[2]{\dhp_{#1;#2}}

\newcommand\funddom[1]{\mathcal{D}}

\newcommand\hammingsphere{\Omega}

\newcommand\dualcode{\bot}

\newcommand\cutoff[1]{\mathrm{t}(#1)}
  \newcommand\cutoffonly{\mathrm{t}}
\newcommand\minLoz[1]{m_0}
\newcommand\minCoz[1]{w_0}

\newcommand\dhp{Q}

\newcommand\fixed[1]{\dot{#1}}

\newcommand\cfixed{\fixed{v}}
\newcommand\ccfixed{\fixed{c}}
\newcommand\wtcfixed{\wt(\cfixed)}
\newcommand\invgen[2]{\dhp_{#1,#2;\cfixed}}

\newcommand\rnk[1]{n}

\DeclareMathOperator\minwt{min}

\newcommand\wt{\mathrm{wt}}
\newcommand\isectcode[2]{#1\cap #2}

\newcommand\latshell[2]{{#2}_{#1}}

\newcommand\wecoeff[2]{|\codeshell{#1}{#2}|}

\newcommand\inprodcountc[2]{N_{#1}(C;#2)}
\newcommand\inprodcountcdual[2]{N_{#1}(C;#2)}
\newcommand\codeshell[2]{{#2}_{#1}}
\newcommand\gencodeshell[2]{\mathcal{C}_{#1}(#2)}

\newcommand\inprodc[2]{(#1,#2)}

\newcommand\inlinefrac[2]{#1/#2}

\title[Configurations of Extremal {Type~II} Codes]{Configurations of Extremal {Type~II} Codes}
\thanks{The authors thank Zachary Abel, Henry Cohn, John H.\ Conway,
Benedict~H.\ Gross, Barry Mazur, Ken Ono, Vera Pless, and Eric~M.\ Rains
for helpful comments and suggestions.  During parts of
this research, Elkies was supported by NSF grants DMS-0501029 and DMS-1100511
and by a Radcliffe Fellowship, and  Kominers was supported by the
Harvard College Program for Research in Science and Engineering (PRISE),
a Harvard Mathematics Department Highbridge Fellowship,
an NSF Graduate Research Fellowship, NSF  grant CCF-1216095, an AMS-Simons Travel Grant, and the Harvard Milton Fund.}
\thanks{This work includes a part of the second author's
undergraduate thesis~\cite{Kominers:thesis}.}
\author{Noam D.~Elkies}
\address{Department of Mathematics, Harvard University\newline\indent One Oxford Street\newline\indent Cambridge, MA 02138}
\email{elkies@math.harvard.edu}

\author{Scott Duke Kominers}
\address{Society of Fellows, Harvard University\newline\indent
Rock Center for Entrepreneurship\newline\indent
Harvard Business School\newline\indent
Soldiers Field, Boston, MA 02163}
\email{kominers@fas.harvard.edu}

\date{3.14.15}


\begin{document}
\begin{abstract}
We prove configuration results for extremal Type~II codes, 
analogous to the configuration results of Ozeki and of the second author
for extremal Type~II lattices.  Specifically, we show that for
$n \in \{8, 24, 32, 48, 56, 72, 96\}$ every extremal Type~II code
of length~$n$ is generated by its codewords of minimal weight.
Where Ozeki and Kominers used spherical harmonics and weighted theta functions,
we use discrete harmonic polynomials and harmonic weight enumerators.
Along we way we introduce ``$t\frac12$-designs'' as a discrete analog
of Venkov's spherical designs of the same name.
\end{abstract}
\maketitle

\section{Introduction}\label{sec:intro}
We denote by $\F_2$ the two-element field $\Z/2\Z$.
By a ``code'' we mean a \emph{binary linear code of length $n$},
that is, a linear subspace of $\F_2^{n}$.
For such a code~$C$, and any integer~$w$, we define
$$
\codeshell{w}{C} := \{c \in C: \wt(c)=w \},
$$
where $\wt(c) := |\{ i: c_i = 1 \}|$ is the \emph{Hamming weight}.
Recall that the \emph{dual code} of $C$, denoted $C^\bot$, is defined by
$$
C^\bot :=\{c'\in \F_2^n: \inprodc{c}{c'}=0 \text{ for all } c \in C\},
$$
where $\inprodc{\cdot}{\cdot}$ is the usual bilinear pairing
$\inprodc{x}{y} = \sum_{i=1}^n x_i y_i$ on $\F_2^n$.
Then $C^\bot$ is also linear, with $\dim(C) + \dim(C^\bot) = n$.
A code $C$\/ is said to be \emph{self-dual} if $C=C^\bot$.
Such a code must have $\dim(C) = n/2$; in particular $2 \divs n$.
Because $\inprodc{c}{c} \equiv \wt(c) \bmod 2$, it follows that a self-dual code
$C$\/ is \emph{even}, that is, has $2 \divs \wt(c)$ for every word $c \in C$\/;
equivalently, $\codeshell{w}{C} = \emptyset$ unless $2\divs w$.
A code $C$\/ is said to be \emph{doubly even}, or \emph{of Type~II},
if $4 \divs \wt(c)$ for all $c \in C$\/; equivalently, if $\codeshell{w}{C}=\emptyset$ unless $4\divs w$.

Mallows and Sloane~\cite{MallowsSloane} showed that a Type~II code~$C$\/
of length~$n$ must contain nonzero codewords of weight at most
$4\lfloor n/24 \rfloor + 4$ (see also~\cite[p.~194]{SPLAG}).
If $\codeshell{w}{C} = \emptyset$
for all positive $w < 4\lfloor n/24 \rfloor + 4$,
then $C$\/ is said to be \emph{extremal}, because it is known
that such a code has the largest minimal distance
among all Type~II codes of its length.

In this paper, we prove configuration results for extremal Type~II codes.
Specifically, we show that if $C$ is an extremal Type~II code
of length $n=8, 24, 32, 48, 56, 72$, or~$96$,
then $C$ is generated by its minimal-weight codewords.
Our approach uses the machinery of harmonic weight enumerators
introduced by Bachoc~\cite{Bachoc:binary}
and developed further  in~\cite{Elkies+Kominers:Weighted},
following the approach used to prove analogous results for lattices
in the works of Venkov~\cite{Venkov:32}, Ozeki~\cite{Ozeki:32,Ozeki:48},
and the second author~\cite{Kominers:56+72+96}.

\section{Designs, Extremal Codes, and Discrete Harmonic Polynomials}

Fix a positive integer $n$. For each nonnegative integer $w \leq n$,
denote by $\hammingsphere_w$ the Hamming sphere of radius~$w$
about the origin of $\F_2^n$.  Thus $\hammingsphere_w$ consists
of the $\binom{n}{w}$ binary words of length~$n$ and weight~$w$.  
To such a word $c$ we associate its support $\Sigma_c$,
which is the \hbox{$w$-element} set
$\{ i : 1 \leq i \leq n, \; c_i = 1 \}$.

We use the following definition of a $t$-design in $\hammingsphere_w$,
which neither assumes that $w \geq t$, nor that the design is nonempty.

\begin{definition}\label{def:design}
We say that a subset $\somedesign \subseteq \hammingsphere_w$
is a \emph{$t$-predesign} for an integer $t \geq 0$
if there exists an integer $N=N_t(\somedesign)$ such that
every subset $I \subset \{1, 2, \ldots, n\}$ of cardinality at most~$t$
is contained in exactly $N$ of the sets $\Sigma_c$ with $c \in \somedesign$.
Then a subset $\somedesign \subseteq \hammingsphere_w$
is called a \emph{$t$-design}
if $\somedesign$\/ is a $t'$-predesign for each positive integer $t' \leq t$.
\end{definition}

\begin{remarks}
It is well known that if $w \geq t$, then
$\somedesign$\/ is a \hbox{$t$-design} if and only if
\begin{equation}\label{eq:tdesign-sum}
\binom{n}{w} \sum_{c \in \somedesign} f(c)
 = |\somedesign| \sum_{c \in \hammingsphere_w} f(c)
\end{equation}
for any function $f: \F_2^n \ra \C$ that depends on at most $t$ of the
$n$ coordinates, and that $N_t(\somedesign)$ is given by the formula
\begin{equation}\label{eq:doublecount}
\binom{n}{t} N_t(D) = \binom{w}{t} \, |\somedesign|
\end{equation}
because both sides of (\ref{eq:doublecount}) count
ordered pairs $(I,c)$ such that $|I|=t$, $c \in \somedesign$, and
$I \subseteq \Sigma_c$.  In this case a $t$-predesign $\somedesign$\/
is automatically a \hbox{$t$-design},
but if $w < t$ then every subset of $\hammingsphere_w$ is a \hbox{$t$-predesign}
(with $N_t(D)=0$, still in accordance with (\ref{eq:doublecount})),
so we need the ``predesign'' property also for $t' < t$\/
to assure that a \hbox{$t$-design} is also a \hbox{$t'$-design} for $t' < t$.
Moreover, the only \hbox{$w$-predesigns} in $\hammingsphere_w$ are
$\hammingsphere_w$ itself and $\emptyset$, so once $t \geq w$
it follows that the only \hbox{$t$-designs} in $\hammingsphere_w$ are
$\hammingsphere_w$ itself and~$\emptyset$.
It follows that (\ref{eq:tdesign-sum}) still holds for $t \geq w$.

\end{remarks}

Extremal codes yield designs by the following
important special case of the Assmus-Mattson theorem.
For $n \equiv 0 \bmod 8$, we define $\cutoff{n}$ by
\begin{equation}
\label{eq:defcut}
\cutoff{n}\defeq \begin{cases}
5&n\equiv 0\bmod 24,\\
3&n\equiv 8 \bmod 24,\\
1&n\equiv 16 \bmod 24.\\
\end{cases}
\end{equation}

\begin{theorem}[\cite{assmusmattson}]\label{thm:assmat}
If $C$ is an extremal Type~II code of length $n$,
then $\codeshell{w}{C}$ is a \hbox{$\cutoff{n}$-design} for each~$w$.
\end{theorem}

In \cite[Thm.~7.4]{Elkies+Kominers:Weighted}
we gave a new proof of Theorem~\ref{thm:assmat}
using the discrete harmonic polynomials $\dhp: \F_2^n \rightarrow \C$
introduced by Delsarte \cite{Delsarte:Hahn}, via his characterization of \hbox{$t$-designs}:
\begin{theorem}[{\cite[Thm.~7]{Delsarte:Hahn},
  \cite[Prop.~7.1]{Elkies+Kominers:Weighted}}]\label{prop:equivprop}
A set $\somedesign \subseteq \hammingsphere_w$ is a $t$-design
if and only if
\begin{equation}
\sum_{v\in \somedesign}\dhp(v)=0
\label{eq:sumdhp=0}
\end{equation}
for all nonconstant discrete harmonic polynomials $\dhp$
with $\deg\dhp \leq t$.
\end{theorem}

We note two important corollaries of Theorem~\ref{prop:equivprop}.
The first reorganizes \eqref{eq:sumdhp=0}:

\begin{cor}[{\cite[Thm.~6]{Delsarte:Hahn}, \cite[Cor.~7.3]{Elkies+Kominers:Weighted}}]\label{cor:equivprop2}
A set $\somedesign\subseteq \hammingsphere_w$ is a \hbox{$t$-design}
if and only if
\begin{equation}
\label{eq:alttdes}
\sum_{v\in \somedesign}\dhp(v)
= \frac{\vert\somedesign\vert}{\vert\hammingsphere_w\vert}
  \sum_{v\in \hammingsphere_w}\dhp(v)
\end{equation}
for all discrete harmonic polynomials $\dhp$ with $\deg\dhp \leq t$.
\end{cor}

Note that $\sum_{v\in \hammingsphere_w}\dhp(v)$, and thus also
$\sum_{v\in \somedesign}\dhp(v)$, vanishes unless $\deg Q = 0$.

The second corollary is the special case of \eqref{eq:sumdhp=0} when
$\dhp$ is a \emph{discrete zonal harmonic polynomial},
that is, a discrete harmonic polynomial such that
$Q(v)$ depends only on the weights of $v$ and $\isectcode{v}{\cfixed}$
for some fixed vector~$\cfixed$
(equivalently, $Q(v)$ depends only on $\wt(v)$
and the distance between $v$ and~$\cfixed$).
Given a degree $d$\/ and a fixed $\cfixed \in \F_2^n$,
we showed in~\cite[Sec.~6]{Elkies+Kominers:Weighted}
that there is a one-dimensional space of discrete zonal harmonic polynomials,
generated by
\begin{equation}\label{eq:dhp_def1}
\dzoharmd{d}{\cfixed}(v)
\defeq
 \sum_{k=0}^d(-1)^{k}
  \left(\prod_{\ell=0}^{k-1}
   \frac{(n-\wtcfixed)-(d-\ell-1)}{\wtcfixed-\ell}\right)
   \invgen{d}{k}(v),
\end{equation}
where
\begin{multline*}
\invgen{d}{k}(v) =
\left(
 \sum_{i=0}^k (-1)^i
  \binom{\wt(\isectcode{v}{\cfixed})}{i}
  \binom{\wtcfixed-\wt(\isectcode{v}{\cfixed})}{k-i}
\right)
  \times\\
\left(
 \sum_{i=0}^{d-k}
  \binom{\wt(v)-\wt(\isectcode{v}{\cfixed})}{i}
  \binom{\left(n-\wtcfixed\right)-\left(\wt(v)-\wt(\isectcode{v}{\cfixed})\right)}{d-k-i}
\right).
\label{eq:Qdkv} 
\end{multline*}

\begin{cor}[{\cite[Cor.~7.6]{Elkies+Kominers:Weighted}}]\label{cor:dzhp}
If $\somedesign\subseteq \hammingsphere_w$ is a \hbox{$t$-design} then
\begin{equation}
\sum_{v\in \somedesign}\dzoharmd{d}{\cfixed}(v)=0
\end{equation}
for each positive $d \leq t$ and any $\cfixed\in \F_2^n$.
\end{cor}

The approach to Theorem~\ref{thm:assmat} via discrete harmonic polynomials
is motivated by the fruitful analogy between Type~II codes
and \emph{Type~II lattices}, which are even unimodular Euclidean lattices.
Recall \cite[Ch.~VII]{Serre:course} that
the rank of such a lattice $L$ must be a multiple of~$8$,
and its theta function is a modular form for $\PSL_2(\Z)$.
It follows via a theorem of Siegel \cite{Siegel:extremal} that
if $L$ has rank $n$ then its minimal nonzero norm is
at most $2 \lfloor n/24 \rfloor + 2$
(Mallows--Odlyzko--Sloane~\cite{MallowsOdlyzkoSloane}).
If equality holds then $L$ is said to be \emph{extremal}.
In such a lattice the vectors of each given norm form a
spherical $(2\cutoff{n}+1)$-design.  As in Corollary~\ref{cor:equivprop2},
this means that the sum over those vectors of $P$\/ vanishes
for any nonconstant harmonic polynomial $P$\/ of degree
at most $2\cutoff{n}+1$.  The $(2\cutoff{n}+1)$-design property
is proved by recognizing the sum as the coefficient of
a modular form (a weighted theta function);
our proof of Theorem~\ref{thm:assmat} in \cite{Elkies+Kominers:Weighted}
is analogous, using harmonic weight enumerators of Type~II codes.

In the lattice setting, the modular-forms approach gives additional
information on the configuration of lattice vectors of given norm,
beyond the fact that the configuration is a $(2\cutoff{n}+1)$-design.
Namely, while the sum of a spherical harmonic of degree $2\cutoff{n}+2$
over lattice vectors of a given norm need not vanish
(i.e., those vectors need not constitute a $(2\cutoff{n}+2)$-design),
a spherical harmonic of degree $2\cutoff{n}+4$ \emph{does} sum to~$0$.
(Odd harmonics sum to~$0$ automatically because the design is
spherically symmetric.)   Venkov~\cite{Venkov:32} calls such a
spherical configuration a ``$(2\cutoff{n}+1\frac12)$-design''\kern-.1ex.
In \cite[Prop.~7.5]{Elkies+Kominers:Weighted}
we proved that for an extremal Type~II code $C$\/ each $\codeshell{w}{C}$
satisfies an additional constraint, analogous to the
$(2\cutoff{n}+1\frac12)$-design property of extremal lattices.
We thus introduce parallel terminology in this setting.
Recall (Theorem~\ref{prop:equivprop}) that
$\somedesign \subseteq \hammingsphere_w$ is a \hbox{$t$-design}
if and only if $\sum_{v \in D} Q(v) = 0$ for all nonconstant
discrete harmonic polynomials $Q$\/ of degree at most~$t$.

\begin{definition}\label{def:design+half}
A subset $\somedesign \subseteq \hammingsphere_w$ is said to be
a \hbox{\emph{$t\frac12$-design}} if $\somedesign$\/ is a
\hbox{$t$-design} such that $\sum_{v \in D} Q(v) = 0$ holds in addition
for all discrete harmonic polynomials $Q$\/ of degree~$t+2$.
\end{definition}

Then the result from \cite{Elkies+Kominers:Weighted} can be expressed
as follows:

\begin{theorem}[{\cite[Prop.~7.5]{Elkies+Kominers:Weighted}}]\label{thm:extremeassmat}
Let $\cutoffonly=\cutoff{n}$.
If $C$\/ is an extremal Type~II code of length $n$,
then $\codeshell{w}{C}$ is a \hbox{$\cutoffonly\frac12$-design}
for each~$w$.  In particular, for each $w$ and any $\cfixed\in \F_2^n$,
\begin{equation}
\sum_{v\in \codeshell{w}{C}}\dzoharmd{d}{\cfixed}(v)=0
\label{eq:sumdzhp=0}
\end{equation}
holds for positive $d \leq \cutoffonly$ and also for $d = \cutoffonly + 2$.
\end{theorem}

\section{Configuration Results}

\subsection{Preliminaries.}
For any $\cfixed\in \F_2^n$, any length-$n$ binary linear code~$C$,
and any $j$ ($0\leq j\leq n$), we denote by $\inprodcountc{j}{\cfixed}$
the value
$$
\inprodcountc{j}{\cfixed} \defeq
 \left|
  \left\{
    c\in \gencodeshell{\minwt(C)}{C}\setsep \wt(\isectcode{c}{\cfixed})=j
  \right\}
 \right|.
$$
For $c\in C^\dualcode$, we must have $\inprodcountc{2j'+1}{c}=0$
for all $j'$ with $0\leq j'\leq \lfloor \inlinefrac{n}{2}\rfloor$.

Throughout the remainder of this section,
$C$ denotes a length-$n$ extremal Type~II code, and
$\minCoz{n}\defeq\minwt(C)$ denotes the minimal weight of codewords in $C$.

\begin{lemma}\label{lem:toobig}
If $\ccfixed$ is a minimal-weight representative of the class
$[\ccfixed]\in C/\gencodeshell{\minCoz{n}}{C}$
and $c\in \latshell{\minCoz{n}}{C}$, then
$$
\wt(\isectcode{c}{\ccfixed})\leq \frac{\minCoz{n}}{2}.
$$ 
\end{lemma}
\begin{proof}
If $\wt(\isectcode{c}{\ccfixed})>\inlinefrac{\minCoz{n}}{2}$,
then $[\ccfixed]$ contains a codeword $c+ \ccfixed$ of weight
$$
\wt(c+ \ccfixed) = \wt(c)+\wt(\ccfixed)-2\wt(\isectcode{c}{\ccfixed})
  < \wt(\ccfixed).
$$
This contradicts the minimality of $\ccfixed$ in $[\ccfixed]$.
\end{proof}

\subsection{Extremal Type~II Codes of Lengths $48$ and $72$}\label{sec:48+72c}

We begin with a configuration result for Type~II codes of lengths $n=48, 72$.

\begin{theorem}\label{thm:48+72c}
If $C$ is an extremal Type~II code of length $n=48$ or $72$, then
$$
C = \gencodeshell{\minCoz{n}}{C}.
$$
\end{theorem}
\begin{proof}
We consider the equivalence classes of $C/\gencodeshell{\minCoz{n}}{C}$
and assume for the sake of contradiction that there is some class
$[\ccfixed]\in C/\gencodeshell{\minCoz{n}}{C}$
with minimal-weight representative $\ccfixed$ with $\wt(\ccfixed)=s>\minCoz{n}$.

As $C$ is self-dual, we have $\inprodcountc{2j'+1}{c}=0$
for all $0\leq j'\leq \lfloor \inlinefrac{n}{2}\rfloor$.
Additionally, by Lemma~\ref{lem:toobig}, we must have
$\inprodcountc{2j'}{\ccfixed}=0$ for $j'> \inlinefrac{\minCoz{n}}{4}$.
We now develop a system of equations in the
$$
\frac{\minCoz{n}}{4}+1
$$
variables $\inprodcountc{0}{\ccfixed},\inprodcountc{2}{\ccfixed},\ldots,
  \inprodcountc{\inlinefrac{\minCoz{n}}{2}}{\ccfixed}$.

Combining the $\cutoff{n}+1$ equations of Corollary~\ref{cor:dzhp}
with the equation
\begin{equation}
\label{eq:codesumup}
\inprodcountc{0}{\ccfixed}+\inprodcountc{2}{\ccfixed}+\cdots
 + \inprodcountc{\inlinefrac{\minCoz{n}}{2}}{\ccfixed}
= \wecoeff{\minCoz{n}}{C}
\end{equation}
gives a system of
$$
\cutoff{n}+2> \frac{\minCoz{n}}{4}+1
$$
equations in the variables $\inprodcountc{2j'}{\ccfixed}$
($0\leq j'\leq \inlinefrac{\minCoz{n}}{4}$). 

For $n=48, 72$, the (extended) determinants of these inhomogeneous systems are
\begin{gather}\label{eq:firstdetc1}
 2^{26}3^{5}5^{2}7^{1}11^{2}23^{2}43^{1}47^{1}\left(\frac{ 11 s^3-396 s^2+4906 s-20736}{(s-3) (s-2)^2 (s-1)^3 s^3}\right),\\
2^{42}3^{5}5^{2}7^{2}11^{2}13^{1}17^{3}23^{2}67^{2}71^{1}\left(\frac{ 39 s^4-2600 s^3+67410 s^2-800440 s+3650496}{(s-4) (s-3)^2 (s-2)^3 (s-1)^4 s^4}\right),\label{eq:lastdetc1}
\end{gather}respectively\footnote{
  These determinants were computed using the formula of
  Corollary~\ref{cor:dzhp}.  We omit the equations obtained from
  the zonal spherical harmonic polynomials of the highest degrees
  when there are more than $\frac{\minCoz{n}}{4}+2$ equations obtained
  by this method.
  };
these determinants must vanish, as they are derived from overdetermined systems.
Since equations~\eqref{eq:firstdetc1}--\eqref{eq:lastdetc1}
have no integer roots $s$, we have reached a contradiction.
\end{proof}

\subsection{Extremal Type~II Codes of Lengths At Most $32$}
The approach used to prove Theorem~\ref{thm:48+72c}
may also be applied to show that extremal Type~II codes of lengths
$n=8,24$, and $32$ are generated by their minimal-weight codewords.
In these cases the determinants
\begin{gather*}
2^7 3^1 7^1\left(\frac{ 3 s-10}{(s-1) s}\right),\\
2^{15} 3^{2} 5^{1} 7^{1} 11^{2} 23^{1}\left(\frac{ 7 s^2-98 s+344}{(s-2) (s-1)^2 s^2}\right),\\
2^{17}3^{1}5^{2}7^{1}29^{1}31^{1}\left(\frac{ 7 s^2-126 s+584}{(s-2) (s-1)^2 s^2}\right)
\end{gather*} are obtained; none have integral roots $s$.
We therefore recover the following result.
\begin{theorem}\label{thm:8+24+32c}
If $C$ is an extremal Type~II code of length $n=8,24$, or $32$, then
$C=\gencodeshell{\minCoz{n}}{C}$.
\end{theorem}
Technically, Theorem~\ref{thm:8+24+32c} has been known (if only implicitly),
as the extremal Type~II codes of lengths $n=8,24$, and $32$
have been fully classified
\cite{Pless:GF2, PlessSloane1975, ConwayPless1980, ConwayPless1992}.
Our methods, however, let us prove that the extremal Type~II codes
of these lengths are generated by their minimal codewords
without appeal to the classification results
or to the explicit forms of these codes.

There is no analog of Theorems~\ref{thm:48+72c} and~\ref{thm:8+24+32c}
for extremal Type~II codes of length $n=16$.  Indeed, the extremal Type~II
code with tetrad subcode $d_{16}$ has codewords of weight~$8$
that cannot be obtained as linear combinations of codewords of weight~$4$.
As expected, following the method used to prove Theorem~\ref{thm:48+72c}
in the case $n=16$ yields the determinant
$$
-93184\left(\frac{ s-8}{(s-1) s}\right),
$$
which vanishes for $s=8$.

\subsection{Extremal Type~II Codes of Lengths $56$ and $96$}

Now, we prove an analog of Theorem~\ref{thm:48+72c}
for extremal Type~II codes of lengths $n=56,96$.  

\begin{lemma}\label{lem:irredcomps'}
If $C$ is an extremal Type~II code of length $n$,
and $w>0$ is such that $\codeshell{w}{C}\neq\emptyset$,
then for each $j$ ($1\leq j\leq n$) there exists
$c\in\codeshell{w}{C}$ such that $c_j = 1$.
\end{lemma}
\begin{proof}
By Theorem~\ref{thm:extremeassmat}, $\codeshell{w}{C}$ is a $1$-design.
We then have from Corollary~\ref{cor:equivprop2} that
$$
\sum_{c\in \codeshell{w}{C}}c_j
= \frac{|\codeshell{w}{C}|}{\vert\hammingsphere_w\vert}
  \sum_{v\in \hammingsphere_w}v_j>0.
$$
The result follows immediately.
\end{proof}

We now state and prove the configuration result for
extremal Type~II codes of lengths~$56$ and~$96$.

\begin{theorem}\label{thm:56+96c}
If $C$ is an extremal Type~II code of length $n=56,96$, then
$$
C=\gencodeshell{\minCoz{n}}{C}.
$$
\end{theorem}
\begin{proof}
Seeking a contradiction, we suppose that $\gencodeshell{\minCoz{n}}{C}\neq C$,
and consider $\gencodeshell{\minCoz{n}}{C}^\dualcode$.
We must have
$\gencodeshell{\minCoz{n}}{C}^\dualcode\neq  \gencodeshell{\minCoz{n}}{C}$,
since otherwise we would have $\gencodeshell{\minCoz{n}}{C}= C$\/
by Lemma~\ref{lem:irredcomps'}.  Thus, there is some equivalence class
$[\ccfixed]\in
  (\gencodeshell{\minCoz{n}}{C}^\dualcode)/(\gencodeshell{\minCoz{n}}{C})$
with minimal-weight representative $\ccfixed$ of weight $\wt(\ccfixed)=s>0$.

Corollary~\ref{cor:dzhp} yields $\cutoff{n}+1$ equations
in the variables $\inprodcountcdual{2j'}{\ccfixed}$
($0\leq j'\leq \inlinefrac{\minCoz{n}}{4}$).\footnote{
  Note that the variables $\inprodcountcdual{j}{\ccfixed}$ vanish for
  $j$ not of the form $2j'$ with $0\leq j'\leq \inlinefrac{\minCoz{n}}{4}$,
  as the conclusion of Lemma~\ref{lem:toobig} holds for $[\ccfixed]\in
  (\gencodeshell{\minCoz{n}}{C}^\dualcode)/(\gencodeshell{\minCoz{n}}{C})$.
  }
Combining these equations with \eqref{eq:codesumup},
we obtain a system of $\cutoff{n}+2$ equations in the
$$
\frac{\minCoz{n}}{2}+1< \cutoff{n}+2
$$
variables $\inprodcountcdual{2j'}{\ccfixed}$
($0\leq j'\leq \inlinefrac{\minCoz{n}}{4}$).
For $n=56, 96$, these inhomogeneous systems have
(extended) matrices with determinants
\begin{gather}
-2^{27}3^{7}5^{3}7^{3}11^{1}13^{2}17^{1}53^{1}\left(\frac{(s-16) \left(3 s^3-112 s^2+1368 s-5120\right)}{(s-4) (s-3) (s-2)^2 (s-1)^3 s^3}\right),\label{eq:56detc}\\
-2^{59}3^{9}5^{4}7^{2}11^{2}13^{2}17^{1}19^{1}23^{3}29^{1}31^{2}43^{1}47^{2}89^{2}\cdot S_{96}(s),\label{eq:96detc} 
\end{gather}
where $S_{96}$ is the rational function
$$S_{96}(s)=\left(\frac{(s-24) \left(68 s^5-6936 s^4+289901 s^3-6153306 s^2+65640728
   s-277774080\right)}{(s-6) (s-5) (s-4)^2 (s-3)^3 (s-2)^4 (s-1)^5 s^5}\right).
$$
These determinants must vanish, but the only integral roots of
\eqref{eq:56detc} and \eqref{eq:96detc} are multiples of~$4$.
Therefore, $\gencodeshell{\minCoz{n}}{C}^\dualcode$ is doubly even, and
it follows that $\gencodeshell{\minCoz{n}}{C}^\dualcode$ is self-orthogonal.
Then, $\dim(\gencodeshell{\minCoz{n}}{C}^\dualcode)\leq \inlinefrac{n}{2}$
and so $\dim(\gencodeshell{\minCoz{n}}{C})\geq \inlinefrac{n}{2}$.
We must therefore have $\gencodeshell{\minCoz{n}}{C}=C$.
\end{proof}

\begin{remarks}
After the our results were first circulated in
\cite{Kominers:thesis} and \cite{Elkies+Kominers:Weighted},
Harada~\cite{Harada:5-design} showed the following result that
generalizes our configuration results for extremal Type~II codes
of lengths $24$, $48$, $72$, and $96$.
\begin{theoremnonum}[{\cite{Harada:5-design}}]\label{thm:harada}
If $C$ is an extremal Type~II code of length $24m$ for a positive integer
$m \leq 6$, and $w = 4k$ for some integer $k$ with $m < k < 5m$
and $(m,k) \neq (6,18)$, then $C = \gencodeshell{w}{C}$.
\end{theoremnonum}
\noindent
Taking $k=m+1$ for $m=1,2,3,4$ recovers our results for $n=24m$.
Harada's approach is different from ours, as it uses 
Mendelsohn's relations \cite{Mendelsohn71}.
\end{remarks}

\bibliographystyle{amsalpha}
\bibliography{references}
\end{document}